\documentclass[11pt]{amsart}
\usepackage{amssymb}
\usepackage{graphicx}
\usepackage{amscd}
\usepackage{amsmath}
\usepackage{amsfonts,latexsym,amstext}

\usepackage [T1]{fontenc}
\usepackage [latin1]{inputenc}
\usepackage{color}
\usepackage[all]{xy}

\newtheorem{theorem}{Theorem}[section]
\newtheorem{definition}[theorem]{Definition}
\newtheorem{question}[theorem]{Question}
\newtheorem{example}[theorem]{Example}
\newtheorem{lemma}[theorem]{Lemma}
\newtheorem{proposition}[theorem]{Proposition}
\newtheorem{corollary}[theorem]{Corollary}
\newtheorem{remark}[theorem]{Remark}

\def\o{\otimes^k }

\def\s{\backslash}
\def\et{/ \pi_{k,s} \s}
\def\etd{\s \varepsilon_{k,s} /}

\def\s{\backslash}
\def\os{\otimes^{k,s}}

\def\Q{\mathcal{Q}}

\begin{document}

\title[Geometry, $M$-ideals and unique norm preserving extensions]{Geometry of integral polynomials, $M$-ideals and unique norm preserving extensions}

\author{Ver\'{o}nica Dimant}
\author{Daniel Galicer}
\author{Ricardo Garc\'{\i}a}

\thanks{The first author was partially supported by CONICET PIP 0624. The second author was partially supported by UBACyT Grant X218, CONICET PIP 0624 and a doctoral fellowship from CONICET. The third author has been supported in part by Project MTM2010-20190-C02-01 and Junta de Extremadura CR10113 ``IV Plan Regional I+D+i, Ayudas a Grupos de Investigaci\'{o}n''.}

\address{Departamento de Matem\'{a}tica, Universidad de San
Andr\'{e}s, Vito Dumas 284, (B1644BID) Victoria, Buenos Aires,
Argentina and CONICET.} \email{vero@udesa.edu.ar}

\address{Departamento de Matem\'{a}tica - Pab I,
Facultad de Cs. Exactas y Naturales, Universidad de Buenos Aires,
(C1428EGA) Buenos Aires, Argentina and CONICET.} \email{dgalicer@dm.uba.ar}

\address{Departamento de Matem\'{a}ticas, Universidad de Extremadura, Avenida de Elvas,
(ES-06071) Badajoz, Spain}
\email{rgarcia@unex.es}


\begin{abstract}%

We use the Aron-Berner extension to prove that the set of extreme points of the unit ball of the space of integral polynomials over a real Banach space $X$ is $\{\pm \phi^k: \phi \in X^*, \| \phi\|=1\}$. With this description we show that, for real Banach spaces $X$ and $Y$, if $X$ is a non trivial $M$-ideal in $Y$, then $\widehat\bigotimes^{k,s}_{\varepsilon_{k,s}} X$ (the $k$-th symmetric tensor product of $X$ endowed with the injective symmetric tensor norm) is \emph{never} an $M$-ideal in $\widehat\bigotimes^{k,s}_{\varepsilon_{k,s}} Y$. This result marks up a difference with the behavior of non-symmetric tensors since, when $X$ is an $M$-ideal in $Y$, it is known  that $\widehat\bigotimes^k_{\varepsilon_k} X$ (the $k$-th tensor product of $X$ endowed with the injective tensor norm) is an $M$-ideal in $\widehat\bigotimes^k_{\varepsilon_k} Y$.
Nevertheless, if $X$ is Asplund, we prove that every integral $k$-homogeneous polynomial in $X$ has a unique extension to $Y$ that preserves the integral norm. We explicitly describe this extension.

We also give necessary and sufficient conditions (related with the continuity of the Aron-Berner extension morphism) for a fixed  $k$-homogeneous polynomial $P$ belonging to a maximal polynomial ideal $\Q(^kX)$ to have a unique norm preserving extension to $\Q(^kX^{**})$. To this end, we study the relationship between the bidual of the symmetric tensor product of a Banach space and the symmetric tensor product of its bidual and show (in the presence of the BAP) that both spaces have `the same local structure'.
Other applications to the metric and isomorphic theory of symmetric tensor products and polynomial ideals are also given.
\end{abstract}

\maketitle

\section*{Introduction}

The  study of  extreme points of dual unit balls and the presence of $M$-ideal structures showed to be very useful tools in the theory of Banach spaces, leading to a better understanding of the geometry of the spaces involved.

Ruess-Stegall \cite{RS}, Ryan-Turett \cite{RT}, Boyd-Ryan \cite{BR},  Dineen \cite{Di}   and Boyd-Lassalle \cite{BL}   in their investigations studied the extreme points of the unit ball of the space of (integral) polynomials defined on a Banach space.
On the other hand, a number of authors have examined $M$-ideal structures in tensor products, operator spaces, spaces of polynomials or Banach algebras,  see e.g.   D. Werner \cite{W,W1}, Dimant \cite{D}, V. Lima \cite{L} and Harmand-Werner-Werner \cite{HWW} (see also the references therein).

Motivated by the increasing interest in the theory of homogeneous polynomials and symmetric tensor products, we study the  extreme points of the unit ball of a space of integral polynomials, the existence of an $M$-ideal structure in the symmetric injective tensor product and unique norm preserving extensions for integral polynomials (and also for polynomials belonging to other ideals).

In 1972 Alfsen and Effros \cite{AE} introduced the notion of an $M$-ideal in a Banach space. The presence of an $M$-ideal $X$ in a Banach space $Y$ in some way expresses that the norm of $Y$ is a sort of maximum norm (hence the letter $M$). To be more
precise, a subspace $X$ of a Banach space $Y$ is an $M$-ideal in $Y$ if its annihilator, $X^\perp$, is $\ell_{1}$-complemented (then
we may write $Y^*=X^\perp \oplus_1 X^*$, see Section 1).

As it is quoted in the book written by Harmand, Werner and Werner \cite{HWW}:
``The fact that $X$ is an $M$-ideal in $Y$ has a strong impact on both $Y$ and $X$ since there are
a number of important properties shared by $M$-ideals, but not by arbitrary subspaces''.
One of the interesting properties shared by $M$-ideals is the following: if $X$ is an $M$-ideal in $Y$ then every linear functional defined in $X$ has a \emph{unique} norm preserving extension to a functional in $Y^*$ \cite[Prop. I.1.12]{HWW}.

As a consequence of \cite[Prop. VI.3.1]{HWW} we know that if $X$ is an $M$-ideal in $Y$, then $\widehat\bigotimes^k_{\varepsilon_k} X$ (the $k$-th tensor product of $X$ endowed with the injective tensor norm) is an $M$-ideal in $\widehat\bigotimes^k_{\varepsilon_k} Y$. Therefore, every integral $k$-linear form in $X$ has a unique  extension to $Y$ that preserves the integral norm.

Most of the results of the theory of tensor products and tensor norms have their natural analogue in the symmetric context (i.e. in the theory of symmetric tensor products); so one should expect that whenever $X$ is a non trivial $M$-ideal in $Y$, then $\widehat\bigotimes^{k,s}_{\varepsilon_{k,s}} X$ (the $k$-th symmetric tensor product of $X$ endowed with the injective symmetric tensor norm) would be an $M$-ideal in $\widehat\bigotimes^{k,s}_{\varepsilon_{k,s}} Y$. Surprisingly,  we will see in Theorem \ref{no es M-ideal} that, for real Banach spaces, this never can happen.
To prove this, we make use of a characterization of the extreme points of the unit ball of the space of integral polynomials over real Banach spaces, which is interesting in its own right.

In \cite{BL} Boyd and Lassalle proved that if $X$ is a real Banach space, $X^*$ has the
approximation property and  $\widehat{\otimes}^{k,s}_{\varepsilon_{k,s}}X$ does not contain a copy of $\ell_{1}$, then the set of extreme points of the unit ball of the space of integral $k$-homogeneous polynomials over $X$ is $\{\pm \phi^k: \phi \in X^*, \| \phi\|=1\}$. We will show in Theorem \ref{Extreme points} that all the additional hypotheses of their result can be removed.

Even though the  $M$-structure for symmetric tensors fails, one may wonder whether the consequence about unique norm  preserving extensions holds. That is, being $X$  a non trivial $M$-ideal in $Y$, has every integral $k$-homogeneous polynomial in $X$  a unique extension to $Y$ that preserves the integral norm? We will give in Theorem \ref{unique extension integrales} a positive answer for the case of $X$ being an Asplund space and describe explicitly this unique extension.
In particular, if $X$  is an $M$-ideal in its bidual $X^{**}$ then every integral $k$-homogeneous polynomial in $X$ has a unique extension to $X^{**}$ with the same integral norm.

We will also examine the following related question: Let $\Q$ be a maximal polynomial ideal and  let $P$ be a fixed polynomial belonging to $\Q(^kX)$, under what conditions do we have a \emph{unique} norm preserving extension of $P$ to the bidual $X^{**}$? Since the Aron-Berner extension preserves the ideal norm for maximal polynomial ideals \cite{CarGal}, the question can be rephrased in the following way: When is the Aron-Berner extension the \emph{only} norm preserving extension (for a given polynomial) in $\Q$?

This was addressed in \cite{ABCh} for the ideal of continuous homogeneous polynomials. We will see in Section 4 necessary and sufficient conditions for this to happen that are related with the continuity of the Aron-Berner extension morphism.

To this end, given a symmetric tensor norm of order $k$, $\beta_k$,  we examine the local geometry of the  the bidual of the symmetric tensor product of a Banach space, $\left(\widehat{\bigotimes\nolimits}^{k,s}_{\beta_k}X\right)^{**}$, and the symmetric tensor product of its bidual $\widehat{\bigotimes\nolimits}^{k,s}_{\beta_k}X^{**}$. We introduce a canonical application $\Theta_{\beta_k}$ from $\widehat{\bigotimes\nolimits}^{k,s}_{\beta_k}X^{**}$ to $\left(\widehat{\bigotimes\nolimits}^{k,s}_{\beta_k}X\right)^{**}$, in order to study if these spaces  have the same local structure in the sense of the principle of local reflexivity. Some isomorphic properties are derived from this relationship. For example, in Theorem \ref{loc-compl} we show that, if $X^{**}$ has the bounded approximation property, then $\Theta_{\beta_k}$ embeds $\widehat{\bigotimes}^{k,s}_{\beta_k}X^{**}$ as a locally complemented subspace of
$\left(\widehat{\bigotimes}^{k,s}_{\beta_k}X\right)^{**}$. Equivalently, $\mathcal{P}_{\beta_k}(^kX^{**})$ (the maximal ideal of $\beta_k$-continuous $k$-homogeneous polynomials over $X^{**}$) is a complemented subspace of $\mathcal{P}_{\beta_k}(^kX)^{**}$. This  extends results of Jaramillo-Prieto-Zalduendo \cite[Cor. 3]{JPZ} and Cabello-Garc\'ia \cite[Th. 2]{CG}.

We will also find conditions to ensure the existence of a canonical isomorphism between these two spaces, i.e. the $Q$-reflexivity for the $\beta_k$ norm introduced by Aron and Dineen \cite{AD}.\\

The article is organized as follows. In Section 1 we give some preliminary background. All the results mentioned above regarding the injective symmetric tensor norm and the ideal of integral polynomials are presented in Section \ref{integral polynomials on M-ideals}.
In Section 3 we describe and study some properties of the mapping $\Theta_{\beta_k}$ which relates $\widehat{\bigotimes\nolimits}^{k,s}_{\beta_k}X^{**}$ and
$\left(\widehat{\bigotimes\nolimits}^{k,s}_{\beta_k}X\right)^{**}$. In Section 4 we study necessary and sufficient conditions that assure that a given $k$-homogeneous polynomial $P$ belonging to a maximal polynomial ideal $\Q(^kX)$ has a unique norm preserving extension to $\Q(^kX^{**})$. In Section 5 we investigate under which circumstances the mapping $\Theta_{\beta_k}$ (equivalently $(\Theta_{\beta_k})^*$) becomes an isomorphism; providing thus an isomorphism between $\mathcal{P}_{\beta_k}(^kX)^{**}$ and  $\mathcal{P}_{\beta_k}(^kX^{**})$.

We refer to \cite{F,FH} for the background about symmetric  tensor products and polynomial ideals and to \cite{HWW} for the theory of $M$-ideals.

\bigskip

\section{Preliminaries}

Throughout the paper $X$ and $Y$ will be real or complex Banach spaces, $X^{*}$ will denote the dual space of $X$, $B_X$ will be the closed unit ball of $X$ and $S_X$ will stand for the unit sphere.
The canonical inclusion from $X$ to $X^{**}$ will be denoted by $\kappa_X$. We will also note by $FIN(X)$  the class of all finite dimensional subspaces of $X$.

We will use the notation $\bigotimes^k X$ for the $k$-fold
tensor product of $X$.
For simplicity, $\otimes^{k} x$ will stand for the elementary tensor $x \otimes  \overset k \cdots \otimes
x$.
The subspace of $\bigotimes^k X$ consisting of all
tensors of the form $\sum_{j=1}^r \lambda_j \otimes^{k} x_j$, where $\lambda_j$ is a scalar and $x_j\in X$ for all $j$, is called the
\emph{symmetric $k$-fold tensor product of $X$} and is denoted by $\bigotimes^{k,s} X$. When $X$ is a vector space over $\mathbb C$, the scalars are not needed in the previous expression.

Given a  continuous operator $T\colon X\to Y$,
\emph{the symmetric $k$-tensor power of $T$} (or the \emph{tensor operator of $T$}) is the mapping from $\bigotimes^{k,s}
X$ to $\bigotimes^{k,s}Y$ defined by
$$
\Big( \os T \Big)(\otimes^k x) = \otimes^k (Tx)
$$
on the elementary tensors and extended by linearity.

For a $k$-fold symmetric tensor $v\in \bigotimes^{k,s} X$, the \emph{symmetric projective norm} of $v$ is given by
$$
\pi_{k,s}(v)=\inf\left\{\sum_{j=1}^r |\lambda_j| \| x_j\|^k\right\},
$$
where the infimum is taken over
all the representations of $v$ of the form  $\sum_{j=1}^r \lambda_j\otimes^k x_j$.

On the other hand,
\emph{the symmetric
injective norm} of $v$ is defined by
$$
\varepsilon_{k,s}(v)= \sup_{\phi\in B_{X^{*}}}\left|\sum_{j=1}^r
 \lambda_j\phi(x_j)^k\right|,
$$
where $\sum_{j=1}^r \lambda_j\otimes^k x_j$ is any fixed representation of $v$.
For properties of these two classical norms ($\varepsilon_{k,s}$ and $\pi_{k,s}$) see \cite{F}.

\medskip

Symmetric tensor products linearize homogeneous polynomials. Recall that a function $P\colon X\to \mathbb{K}$ is said to be a (continuous)
\emph{$k$-homogeneous polynomial} if there exists a
(continuous) symmetric $k$-linear form  $$ A :{X\times \overset k\cdots \times X} \to \mathbb{K}$$ such that
$P(x)= A(x,\ldots,x)$ for all $x\in X$.
In this case, $A$ is called the symmetric $k$-linear form associated to $P$ and it is usually denoted by $\overset\vee P$.
Continuous
$k$-homogeneous polynomials are those bounded in the unit ball, and the norm of such $P$ is given by $$\|P\|=\sup_{\|x\|\le 1}|P(x)|.$$ If we denote
by $\mathcal{P}(^kX)$ the Banach space of all continuous
$k$-homogeneous polynomials on $X$  endowed with the sup norm, we have the isometric identification \begin{equation}\label{dualidad-pol-tensor}\mathcal P(^kX) \overset 1 = \Big(\bigotimes\nolimits^{k,s}_{\pi_{k,s}} X\Big)^{*}.\end{equation}

\medskip
We say that  $\beta_k$ is a \emph{symmetric tensor norm  of order $k$} (s-tensor norm) if $\beta_k$ assigns to each normed space $X$ a norm $\beta_k \big(\; . \;; \bigotimes^{k,s} X \big)$ on the $k$-fold symmetric tensor product $\bigotimes^{k,s} X$ such that
\begin{enumerate}
\item $\varepsilon_{k,s} \leq \beta_k \leq \pi_{k,s}$ on $\bigotimes^{k,s} X$.
\item $\| \os T :   \bigotimes^{k,s}_{\beta_k} X \to \bigotimes^{k,s}_{\beta_k} Y \| \leq \|T\|^k$ for each operator $T \in \mathcal{L}(X,Y)$.
\end{enumerate}
Condition $(2)$ will be referred to as the \emph{``metric mapping property''}. We denote by ${\bigotimes}^{k,s}_{\beta_k} X$ the tensor product ${\bigotimes}^{k,s}X$ endowed with the norm $\beta_k \big(\; . \;; \bigotimes^{k,s} X \big)$, and we write
$\widehat{\bigotimes}^{k,s}_{\beta_k} X$ for its completion.

An s-tensor norm $\beta_k$ is called \emph{finitely generated} if for every normed space $X$ and $v \in \bigotimes^{k,s} X$, we have: $$ \beta_k (v, \bigotimes\nolimits^{k,s}X) = \inf \left\{ \beta_k(v, \bigotimes\nolimits^{k,s}M) : M \in FIN(X),\; v \in \bigotimes\nolimits^{k,s}M \right\}.$$

From now on, every s-tensor norm considered in the article will be finitely generated.

If $\beta_k$ is an s-tensor norm of order $k$, then the \emph{dual tensor norm $\beta_k^{*}$} is defined on FIN (the class of finite dimensional spaces) by
\begin{equation}\label{defi prima dim finita}   \bigotimes\nolimits^{k,s}_{\beta_k^{*}} M :\overset 1 = \Big( \bigotimes\nolimits^{k,s}_{\beta_k} M^{*}\Big)^{*}\end{equation}
and on NORM (the class of normed spaces) by
$$ \beta_k^{*} \Big( v, \bigotimes\nolimits^{k,s} X \Big) : = \inf \left\{ \beta_k^{*} (v, \bigotimes\nolimits^{k,s} M ) : v \in \bigotimes\nolimits^{k,s} M \right\},$$
the infimum being taken over all the finite dimensional subspaces $M$ of $X$ whose symmetric tensor product contains $v$.

\bigskip
Since  any s-tensor norm satisfies $\beta_k \leq \pi_{k,s}$, we have a dense inclusion $$\bigotimes\nolimits^{k,s}_{\beta_k}  X \hookrightarrow  \bigotimes\nolimits^{k,s}_{\pi_{k,s}} X. $$ As a consequence, any $P\in \big(\bigotimes\nolimits^{k,s}_{\beta_k} X\big)^{*}$ can be thought as a $k$-homogeneous polynomial on $X$. Different s-tensor norms $\beta_k$ give rise, by this duality, to different classes of polynomials.

We will say that $\beta_k$ is \emph{projective} if, for every metric surjection $Q: X \overset 1 \twoheadrightarrow Y$, the tensor product operator
$$ \os Q : \bigotimes\nolimits^{k,s}_{\beta_k} X \to \bigotimes\nolimits^{k,s}_{\beta_k} Y$$
is also a metric surjection.
On the other hand we will say that $\alpha$ is \emph{injective} if, for every  $I: X \overset 1 \hookrightarrow Y$ isometric embedding, the tensor product operator
$$ \os I : \bigotimes\nolimits^{k,s}_{\beta_k} E \to \bigotimes\nolimits^{k,s}_{\beta_k} F,$$
is an isometric embedding.

The two extreme s-tensor norms, $\pi_s$ and $\varepsilon_s$, are examples of the last two definition: $\pi_s$ is projective and $\varepsilon_s$ is injective.

The projective and injective associates (or hulls) of a symmetric tensor norm $\beta_k$ will be denoted, by extrapolation of the 2-fold full case, as $\s \beta_k /$ and $/ \beta_k \s$ respectively. The projective associate of $\beta_k$ will be the (unique) smallest projective tensor norm greater than $\beta_k$. Following some ideas from \cite[Th. 20.6.]{DF} we have
$$ \otimes^{k,s} Q_X \colon \bigotimes\nolimits^{k,s}_{\beta_k} \ell_1(X) \overset 1 \twoheadrightarrow  \bigotimes\nolimits^{k,s}_{\s \beta_k /}  X,$$ where $Q_X : \ell_1(B_X) \twoheadrightarrow X$ is the canonical quotient map. We say that $\beta_k$ is projective if $\beta_k=\s \beta_k /$.

The injective associate of $\beta_k$ will be the (unique) greatest injective tensor norm smaller than $\beta_k$.  As in \cite[Th. 20.7]{DF} we get,
$$ \otimes^{k,s} I_X \colon  \bigotimes\nolimits^{k,s}_{/ \beta_k \s} X  \overset 1 \hookrightarrow  \bigotimes\nolimits^{k,s}_{\beta_k} \ell_{\infty}(B_{X^*}),$$  where $I_X$ is the canonical embedding. We say that $\beta_k$ is injective if $\beta_k=/ \beta_k \s$.

The following duality relations for an s-tensor norm $\beta_k$ are easily obtained (see \cite{CarGal-five})
$$  (/ \beta_k \s)^* = \s \beta_k^* /, \; \; \;   (\s \beta_k /)^* = / \beta_k^* \s.$$

Let us recall some definitions on the theory of Banach polynomial ideals
\cite{Flo02(On-ideals)}. A \emph{Banach ideal of continuous scalar valued
$k$-homogeneous polynomials} is a pair
$(\mathcal{Q},\|\cdot\|_{\mathcal Q})$ such that:
\begin{enumerate}
\item[(i)] $\mathcal{Q}(^kX)=\mathcal Q \cap \mathcal{
P}(^kX)$ is a linear subspace of $\mathcal{P}(^kX)$ and $\|\cdot\|_{\mathcal Q(^kX)}$ (the restriction of $\|\cdot\|_{\mathcal Q}$ to $\mathcal{Q}(^kX)$) is a norm which makes
$(\mathcal{Q}(^kX),\|\cdot\|_{\mathcal Q(^kX)})$ a Banach space.

\item[(ii)] If $T\in \mathcal{L} (X_1,X)$, $P \in \mathcal{Q}(^kX)$ then $P\circ T\in \mathcal{Q}(^kX_1)$ and $$ \|
P\circ T\|_{\mathcal{Q}(^kX_1)}\le  \|P\|_{\mathcal{Q}(^kX)} \| T\|^k.$$

\item[(iii)] $z\mapsto z^k$ belongs to $\mathcal{Q}(^k\mathbb K)$
and has norm 1.
\end{enumerate}

Let $(\mathcal{Q},\|\cdot\|_{\mathcal Q})$ be a Banach ideal of continuous scalar valued
$k$-homogeneous polynomials and, for $P \in \mathcal{
P}(^kX)$, define
$$\|P\|_{\Q^{max}(^kX)}:= \sup \{ \|P|_M\|_{\Q(^kM)} : M \in FIN(X) \} \in [0, \infty].$$
The \emph{maximal hull} of $\mathcal{Q}$ is the ideal given by $ \mathcal{Q}^{max} := \{P \in \mathcal{
P}(^kX) : \|P\|_{\Q^{max}} < \infty \}$. An ideal $\mathcal{Q}$ is said to be \emph{maximal} if $\mathcal{Q} \overset{1}{=} \mathcal{Q}^{max}$. The \emph{minimal hull} of $\mathcal{Q}$ is the composition ideal $ \mathcal{Q}^{min} := \mathcal{Q}\circ \overline{\mathcal F}$, where $\overline{\mathcal F}$ stands for the ideal of approximable operators, with the usual composition norm. An ideal $\mathcal{Q}$ is said to be minimal if $\mathcal{Q} \overset{1}{=} \mathcal{Q}^{min}$.

By \cite{FH}, a maximal (scalar-valued) ideal of $k$-homogeneous
polynomials is dual to a symmetric tensor product endowed with a
finitely generated s-tensor norm of order $k$. So, for $\beta_k$ a
finitely generated s-tensor norm of order $k$, we denote by
$\mathcal{P}_{\beta_k}$  the polynomial ideal dual to this tensor
norm. That is, for a Banach space $X$,
$$
\mathcal{P}_{\beta_k}(^kX)=\Big(\widehat{\bigotimes\nolimits}^{k,s}_{\beta_k}X\Big)^*.
$$

The maximal ideal dual to the symmetric tensor norm $\beta_k=\varepsilon_{k,s}$  is the ideal of integral polynomials $\mathcal P_I$. A  polynomial $P\in\mathcal P(^kX)$ is \emph{integral} if  there exists a regular  Borel
measure $\mu$, of bounded variation on $(B_{X^*}, w^*)$ such that
$$ P(x) =\int_{B_{X^*}} \phi(x)^k\ d\mu(\phi),$$ for all $x\in X.$
The integral norm of  $P$ is given by
$$\| P\|_{\mathcal{P}_{I}(^kX)} =\inf\left\{|\mu|(B_{X^*})\right\},$$
where the infimum is taken over all the measures $\mu$ representing
$P$ as above.

The minimal hull of $\mathcal P_I$ is the ideal of nuclear polynomials $\mathcal P_N$.  A polynomial $P\in\mathcal P(^kX)$ is
\emph{nuclear} if it can be written as $P(x)=\sum_{j=1}^\infty
\lambda_j\phi_j(x)^k$, where $\lambda_j\in\mathbb K$, $\phi_j\in X^*$ for
all $j$ and $\sum_{j=1}^\infty |\lambda_j| \Vert \phi_j\Vert ^k<\infty $. The nuclear norm of $P$ is
\[
\Vert P\Vert _{\mathcal P_N(^kX)}=\inf \left\{\sum_{j=1}^\infty |\lambda_j| \Vert \phi_j\Vert ^k\right\},
\]
where the infimum is taken over all the representations of $P$ as
above.

Aron and Berner showed in~\cite{AB} how to extend continuous polynomials defined on a Banach space $X$ to the bidual $X^{**}$. Given a continuous $k$-homogeneous polynomial $P : E \to
\mathbb{K}$ the `Aron-Berner' extension $\overline{P}$ of $P$ is given by means of the corresponding symmetric $k$-linear form $A$. For a $k$-tuple $(z_1, \dots , z_k) \in X^{**} \times \dots \times X^{**}$, consider
 $$\overline A(z_1, \dots , z_k) := \lim_{i_1} \dots \lim_{i_k} A(x_{i_1},... ,x_{i_k}),$$
where each $(x_{i_j})$ is a net in $X$ which converges to $z_j$ in the weak* topology ($j=1, \dots, k$).  The \emph{Aron-Berner extension} of $P$ is defined by $$\overline{P}(z) := \overline{A}(z,... ,z).$$ Although
the definition of $\overline{A}(z_1, \dots, z_k)$ depends on the order in which one calculates the limits,
the definition of the extended polynomial $\overline{P}$  is independent of the
order used (see \cite{LibroDi,Za,Za1} for further details and properties of this extension).

In 1989 Davie and Gamelin \cite{DavieGamelin98} proved that the Aron-Berner extension preserves the uniform norm. In other words, they proved that this extension is a `real' Hahn-Banach extension.
Recently, Daniel Carando the second author \cite{CarGal} extended this result for maximal and minimal polynomial ideals. More precisely, they showed that every maximal or minimal ideal of
$k$-homogeneous polynomials $\Q$ is closed under the Aron-Berner extension (i.e. for every Banach space $X$ and every polynomial $P$ in $\Q(^kX)$, the Aron-Berner extension $\overline{P}$ is in $\Q(^kX^{**})$). Moreover, the Aron-Berner extension morphism $AB: \Q(^kX) \to \Q(^kX^{**})$ given by $P \mapsto \overline{P}$ is an isometry for every Banach space $X$:
\begin{equation}\label{ab isometria}
\|P\|_{\Q(^kX)}=\|\overline{P}\|_{\Q(^kX^{**})}.
\end{equation}

Given a polynomial ideal $\Q$ closed under the Aron-Berner extension and a continuous linear morphism $s: X^{*} \to Y^{*}$, we can construct the following mapping
$ \overline{s}: \Q(^kX) \to \Q(^kY)$ given by $$\overline{s}(P) := \overline{P} \circ s^{*} \circ J_Y,$$ where $J_Y : Y \to Y^{**}$ is the canonical inclusion.
The mapping $\overline{s}$ is referred as `the extension morphism' of $s$. If $s$ is an isomorphism then $\overline{s}$ is also an isomorphism. Moreover, if $s$ is an isometry and the Aron-Berner extension morphism $AB: \Q(^kX) \to \Q(^kX^{**})$ is an isometric mapping (for example if $\Q$ is maximal or minimal) then it is easy to see that $\overline{s}$ is also an isometry.
For more properties and details about $\overline{s}$ see \cite{LZark,Za}.

\medskip
Recall that a Banach space $X$ is said to have the \emph{approximation property} (AP) if, for every absolutely
convex compact set $K$ and every $\varepsilon > 0$ there is a finite rank operator $T \in \mathcal{L}(E, E)$ with $\|Tx - x \| < \varepsilon$ for all $x \in K$.

The bounded approximation property is a version of this property with control of the norms of the finite rank operators involved. It can be defined equivalently in the following way.  The space $X$ is said to have the \emph{$\lambda$-approximation
property} ($\lambda$-AP) if there is a net $T_\gamma$ of
finite rank operators, with $\|T_\gamma\|\leq\lambda$, such that
$\lim_\gamma \|T_\gamma(x)-x\|=0$ for all $x\in X$. A space
having the $\lambda$-AP for some finite $\lambda$ is said to
have the \emph{bounded approximation property} (BAP). Also, the $1$-AP is called \emph{metric approximation property} (MAP).

\medskip

A closed subspace $X$ of a Banach space $Y$ is an \emph{$M$-ideal} in
$Y$ if $Y^*=X^\perp \oplus_1 X^\sharp$,
where $X^\perp$ is the annihilator of $X$ and $X^\sharp$ is a closed
subspace of $Y^*$. Since $X^\sharp$ can be (isometrically) identified with $X^*$,
it is usual to denote $Y^*=X^\perp \oplus_1 X^*$. However, we often prefer to state explicitly the isometric mapping $s:X^*\to Y^*$, thus obtaining the decomposition $Y^*=X^\perp \oplus_1 s(X^*)$.
The space $X$ is said to be \emph{$M$-embedded} if $X$ is an $M$-ideal in its bidual $X^{**}$.

A Banach space $X$ is \emph{Asplund} (or a strong differentiability space) if every separable subspace $S$ of $X$ has separable continuous dual space $S^*$, or equivalently if its dual space $X^*$ has the Radon  Nikod\'{y}m property \cite{LibroDi,HWW}. Recall that every $M$-embedded space is Asplund \cite[III.3.1]{HWW}.

\medskip
A point
$x \in B_{X}$ is said to be a real (complex) \emph{extreme point} whenever
$\{x + \zeta y : |\zeta | \le 1, \zeta \in \mathbb{R}\} \subset B_{X}$ for $y\in X$ yields $y = 0$ (resp. $\zeta \in \mathbb{C}$). In complex Banach spaces, it is easy to check that every  extreme point of $B_{X}$ is also a complex extreme point. The converse however is not true, since, for instance, every point of $S_{\ell_{1}}$  is a complex extreme point of $B_{\ell_{1}}$.  We denote by $Ext (B_X)$ the set of  real extreme points of the ball $B_X$.
When $X$ is an $M$-ideal in
$Y$, we have the following equality for the sets of extreme points of the
 unit balls \cite[Lemma I.1.5]{HWW}:
$$
Ext(B_{Y^*})=Ext(B_{X^\perp})\cup Ext(B_{X^*}).
$$

\section{Integral polynomials on $M$-ideals} \label{integral polynomials on M-ideals}

If $X$ is an $M$-ideal in $Y$ then, by \cite[Prop. VI.3.1]{HWW}, the associativity of the $\varepsilon$-norm and the transitivity of $M$-ideals, it results that $\widehat\bigotimes_{\varepsilon_k}^{k} X$ is an $M$-ideal in $\widehat\bigotimes_{\varepsilon_k}^{k} Y$. This clearly implies that any $k$-linear integral form on $X$ (being an element of the dual of $\widehat\bigotimes_{\varepsilon_k}^{k} X$) has a unique (integral) norm preserving extension to a $k$-linear integral form on $Y$.

The intuition leads us to think that the same happens in the symmetric case. That is, if $X$ is an $M$-ideal in $Y$ then  $\widehat\bigotimes_{\varepsilon_{k,s}}^{k,s} X$ should be an $M$-ideal in $\widehat\bigotimes_{\varepsilon_{k,s}}^{k,s} Y$ and any integral $k$-homogeneous polynomial on $X$ should have a unique (integral) norm preserving extension to an integral $k$-homogeneous polynomial on $Y$. But things are not always as we expected them to be.

We will see in this section how some properties could change completely and others remain the same in the symmetric case. For real Banach spaces the $M$-ideal condition is {\bf never} (except trivial cases) maintained for symmetric injective tensor products. On the other hand, on $M$-embedded spaces, both in the real and complex settings,  the unique norm preserving extension of integral polynomials holds.

The negative result for real Banach spaces derives from a characterization of the extreme points of integral polynomials  which is interesting by itself. We make use of some results and ideas from \cite{BR} and \cite{BL}, pushing things a little more to obtain a general statement.

In \cite{BR} Boyd and Ryan  investigated the set of  extreme points of the unit ball of  $\mathcal{P}_{I}(^kX)$, for $k>1$,  and they showed the following facts:\

 \begin{enumerate}
\item[$(a)$] For a real Banach space $X$, $\{\pm\phi^k: \phi \in S_{X^*}$ and $\phi$ attains its norm $\} \subseteq Ext(B_{\mathcal{P}_{I}(^kX)})$.\

\item[$(b)$] For a real or complex Banach space $X$, $Ext(B_{\mathcal{P}_{I}(^kX)})\subseteq \{\pm \phi^k: \phi \in S_{X^*}\}$ (see also \cite{CD}).

\end{enumerate}

In \cite{BL} Boyd and Lassalle  proved that if $X$ is a real Banach space, $X^*$ has the
approximation property and  $\widehat{\otimes}^{k,s}_{\varepsilon_{k,s}}X$ does
not contain a copy of $\ell_{1}$, then $Ext (B_{\mathcal{P}_{I}(^kX)})$ is in fact the set  $\{\pm\phi^k: \phi \in S_{X^*}\}$. In the following theorem we show that the hypotheses of their result are not necessary.

\begin{theorem}\label{Extreme points} For a real Banach space  $X$ and a positive integer $k>1$, the set of real extreme points of the unit ball of  $\mathcal{P}_{I}(^kX)$ is  $\{\pm \phi^k: \phi \in S_{X^*}\}$.

\end{theorem}

\begin{proof}
Let $\phi \in S_{X^*}$. Since  it is clear that $\phi$ is a norm attaining element of $S_{X^{***}}$,  by the previous comment $(a)$,  $\phi^k$ is an extreme point of the unit ball of $\mathcal{P}_{I}(^kX^{**})$.

Inspired by the proof of Lemma 1 of \cite{BL}, we will use the fact that $Ext(B)\cap A\subseteq Ext(A)$ whenever $A \subseteq B$. Consider the isometric inclusion
\begin{equation*}
AB:\mathcal{P}_{I}(^kX) \overset1\hookrightarrow \mathcal{P}_{I}(^kX^{**})
\end{equation*}
given by the Aron-Berner extension morphism $P \mapsto \overline{P}$. Thus,
\begin{equation*}
Ext (B_{\mathcal{P}_{I}(^kX^{**})})\cap B_{\mathcal{P}_{I}(^kX)}\subseteq Ext(B_{\mathcal{P}_{I}(^kX)}).
\end{equation*}
Finally,
 \begin{center}
$\{\pm\phi^k: \phi \in S_{X^*} \} \subseteq Ext (B_{\mathcal{P}_{I}(^kX^{**})})\cap B_{\mathcal{P}_{I}(^kX)}\subseteq  Ext (B_{\mathcal{P}_{I}(^kX)})\subseteq \{\pm \phi^k: \phi \in S_{X^*}\}$.
 \end{center}

\end{proof}
\begin{remark}\rm
The previous result is not true for complex Banach spaces. Indeed, Dineen \cite[Prop. 4.1]{Di} proves that, if $X$ is a complex Banach space, then $Ext (B_{\mathcal{P}_{I}(^kX)})$ is contained in $\{\phi^k:$  $\phi$ is a complex extreme point of  $B_{X^*}\}$. Let us consider $X$ the complex space $\ell_{1}$. It is clear that $\phi=(0,1,\ldots,1,\ldots)\in  \ell_{\infty}$ is not a complex  extreme point of $B_{\ell_{\infty}}$. Hence,  $\phi^k$ could not be an extreme point of $B_{\mathcal{P}_{I}(^k\ell_{1})}$.\
\end{remark}

\begin{remark}\rm
Although the spaces $\mathcal{P}_{I}(^kX)$ and $\mathcal{L}_{I}(^kX)$  can be isomorphic (for example if $X$ is stable \cite{AF}), they are very different from a geometric point of view since the set  $Ext (B_{\mathcal{L}_{I}(^kX)})$ is equal to $\{ \phi_{1}\phi_{2}\cdots \phi_{k}: \phi_{i} \in Ext (B_{X^*})\}$ (see \cite{RS,BR}).
\end{remark}

The last  characterization of the extreme points of the ball of integral polynomials leads us to show that for a real Banach space $X$, $\widehat\bigotimes_{\varepsilon_{k,s}}^{k,s} X$ could not be an $M$-ideal in $\widehat\bigotimes_{\varepsilon_{k,s}}^{k,s} X^{**}$, unless $X$ is reflexive. As we have already said, this is a big difference with what happens in the non symmetric case where for $X$ an $M$-embedded space it holds that  $\widehat\bigotimes_{\varepsilon_k}^{k} X$ is an $M$-ideal in $\widehat\bigotimes_{\varepsilon_k}^{k} X^{**}$ \cite[Prop. VI.3.1]{HWW}.

\begin{theorem}
If the real Banach space $X$ is not reflexive, then $\widehat\bigotimes_{\varepsilon_{k,s}}^{k,s} X$ could not be an $M$-ideal in $\widehat\bigotimes_{\varepsilon_{k,s}}^{k,s} X^{**}$.
\end{theorem}

\begin{proof}
Suppose that $\widehat\bigotimes_{\varepsilon_{k,s}}^{k,s} X$ is an $M$-ideal in $\widehat\bigotimes_{\varepsilon_{k,s}}^{k,s} X^{**}$. Then we would have:
$$
Ext (B_{\mathcal{P}_{I}(^kX^{**})})= Ext (B_{\mathcal{P}_{I}(^kX)}) \cup Ext \Big(B_{(\widehat\bigotimes^{k,s}_{\varepsilon_{k,s}}X)^\perp}\Big).
$$
By the description of extreme points of integral polynomials of the previous theorem, this equality would imply
$$
Ext \Big(B_{(\widehat\bigotimes^{k,s}_{\varepsilon_{k,s}}X)^\perp}\Big)=\{\pm \phi^k: \phi \in S_{X^{***}}\setminus S_{X^*}\}.
$$
This is not possible since through the decomposition $X^{***}=X^* \oplus X^\perp$ if we choose $\phi\in S_{X^{***}}$ such that $\phi=\phi_1+\phi_2$, with $\phi_1\in X^*$, $\phi_2\in X^\perp$, $\phi_1, \phi_2\not= 0$, then
$\phi\in S_{X^{***}}\setminus S_{X^*}$ but $\phi^k\not\in (\widehat\bigotimes^{k,s}_{\varepsilon_{k,s}}X)^\perp$.
\end{proof}

With almost the same argument (only changing the decomposition $X^{***}=X^* \oplus X^\perp$ by $Y^{*}=X^* \oplus_1 X^\perp$) we derive the following:

\begin{theorem} \label{no es M-ideal}
If $X$ and $Y$ are real Banach spaces and $X$ is a nontrivial $M$-ideal in $Y$, then $\widehat\bigotimes_{\varepsilon_{k,s}}^{k,s} X$ is not an $M$-ideal in $\widehat\bigotimes_{\varepsilon_{k,s}}^{k,s} Y$.
\end{theorem}

\begin{remark}\rm
As it will be stated in Lemma \ref{convergencia}, in a maximal ideal of polynomials, the $w^*$-convergence of a bounded net is equivalent to the pointwise convergence. This implies that if two maximal polynomial ideals  have the same set of extreme points of the unit balls, then they are the same ideal (isometrically). So, from Theorem \ref{Extreme points}, if $\mathcal Q$ is a maximal ideal of $k$-homogeneous polynomials that satisfies that, on a real Banach space $X$, the set of extreme points of its unit ball is $\{\pm \phi^k: \phi \in S_{X^*}\}$, then it should be $\mathcal Q(^kX)=\mathcal P_I(^kX)$.
\end{remark}

We have no idea how symmetric tensor products on $M$-ideals behave in the complex setting.

\begin{question} If $X$ is a nontrivial $M$-ideal in $Y$, $X$ and $Y$ complex Banach spaces, could $\widehat\bigotimes_{\varepsilon_{k,s}}^{k,s} X$ be an $M$-ideal in $\widehat\bigotimes_{\varepsilon_{k,s}}^{k,s} Y$?
\end{question}

\bigskip

Aron, Boyd and Choi \cite[Prop. 7]{ABCh} prove that if $X$ is an $M$-ideal in $X^{**}$ then the Aron-Berner extension is the unique norm preserving extension from $\mathcal P_N(^kX)$ to $\mathcal P_N(^kX^{**})$. Their argument can be easily adapted to the situation of $X$ being an $M$-ideal in $Y$. Recall that in this case the natural inclusion $s:X^*\to Y^*$ induces a canonical isometry $\overline s:\mathcal P_N(^kX)\to\mathcal P_N(^kY)$ (see the explanation at the end of the previous section).

\begin{proposition}
Let $X$ be an $M$-ideal in $Y$ and let $s:X^*\to Y^*$ be the associated isometric inclusion. For each $P\in \mathcal P_N(^kX)$, $\overline s(P)$ is the unique norm preserving extension to $\mathcal P_N(^kY)$.
\end{proposition}

We want to prove a similar statement for integral polynomials. If $X$ is an Asplund space (which always holds when $X$ is an $M$-ideal in $X^{**}$) we will have a positive result. In this case, nuclear and integral polynomials over $X$ coincide isometrically \cite{BR,CD}. So, by the previous proposition, there is only one nuclear norm preserving extension to $Y$. But if $Y$ is not Asplund we could presumably have integral non nuclear extensions of the same integral norm. We will show that this is impossible to happen.

The result is a consequence of the following lemma.

\begin{lemma}\label{lemma2 integrales con extension unica}
Let $X$ be an Asplund space which is a subspace of a Banach space $Y$. Let $Q \in \mathcal P_I(^kY)$. Given $\varepsilon > 0$ there exists $\widetilde{Q} \in \mathcal P_N(^kY)$ such that $Q$ and $\widetilde{Q}$ coincide on $X$ and
$$\|\widetilde{Q} \|_{\mathcal P_N(^kY)} \leq \|Q\|_{\mathcal P_I(^kY)} + \varepsilon.$$
\end{lemma}

\begin{proof}
Since the restriction of $Q$ to $X$ is nuclear, we can take sequences $(\phi_j)_j \subset X^*$ and $(\lambda_j)_j\subset\mathbb{K}$ such that $Q|_X = \sum_{j=1}^\infty \lambda_j\phi_j^k$ and
\begin{align*}
\sum_{j=1}^\infty |\lambda_j|\|\phi_j\|^k & \leq \|Q|_X \|_{\mathcal P_N(^kX)} + \varepsilon \\
& = \|Q|_X \|_{\mathcal P_I(^kX)} + \varepsilon \\
& \leq \|Q \|_{\mathcal P_I(^kY)} + \varepsilon.
\end{align*}
For each $j$, let $\widetilde\phi_j$ be a Hahn-Banach extension of $\phi_j$ to $Y$. If we define $\widetilde{Q} = \sum_{j=1}^\infty \lambda_j\widetilde\phi_j^k$, then $\widetilde{Q}$ coincides with $Q$ in $X$ and $$\|\widetilde{Q} \|_{\mathcal P_N(^kY)} \leq  \sum_{j=1}^\infty |\lambda_j|\|\widetilde\phi_j\|^k =  \sum_{j=1}^\infty |\lambda_j|\|\phi_j\|^k \leq \|Q\|_{\mathcal P_I(^kY)} + \varepsilon.$$
\end{proof}

\begin{theorem} \label{unique extension integrales}
Let $X$ be an Asplund space which is an $M$-ideal in a Banach space $Y$ and let $s:X^*\to Y^*$ be the associated isometric inclusion. If $P\in\mathcal P_I(^kX)$ then the canonical extension $\overline s(P)$ is the unique norm preserving extension  to $\mathcal P_I(^kY)$.
\end{theorem}

\begin{proof}
The argument is modeled on the proof of \cite[Prop. 7]{ABCh}. We include all the steps for the sake of completness.

Let $P \in \mathcal P_I(^kX)$ and suppose there exists a norm preserving extension $Q \in \mathcal P_I(^kY)$ different from $\overline s(P)$.
Pick $y$ a norm one vector in $Y$ such that $0 < \delta = |Q(y) - \overline s(P)(y)|.$

Note that $X \oplus [y]$ is an Asplund space since $X$ also is. So, by Lemma \ref{lemma2 integrales con extension unica} applied to $X \oplus [y]$, there exists $\widetilde{Q} \in \mathcal P_N(^kY)$ such that $Q$ and $\widetilde{Q}$ coincide on $X \oplus [y]$  and
\begin{align*}
\|\widetilde{Q} \|_{\mathcal P_N(^kY)} & \leq \|Q\|_{\mathcal P_I(^kY)} + \frac{\delta}{4} \\
& = \|P\|_{\mathcal P_I(^kX)} + \frac{\delta}{4}.
\end{align*}
Take a nuclear representation of  $\widetilde{Q} = \sum_{j=1}^\infty \lambda_j\phi_j^k$ such that $\sum_{j=1}^\infty |\lambda_j|\|\phi_j\|^k \leq \|P\|_{\mathcal P_I(^kX)} + \frac{\delta}{2}.$
Since $X$ is an $M$-ideal in $Y$ each $\phi_j\in Y^*$ can be written as the sum of $s\big(\phi_j|_X\big)$ and $\phi_j^\bot$.
Moreover, $\|\phi_j \|= \| s\big(\phi_j|_X\big) \| + \|\phi_j^\bot \|$.

Recall that $\widetilde{Q}$ coincides with $P$ on $X$, thus, for every $x \in X$,
$$P(x) =  \sum_{j=1}^\infty \lambda_j\big( s\big(\phi_j|_X\big)(x)  + \phi_j^\bot (x) \big)^k = \sum_{j=1}^\infty \lambda_j\phi_j|_X(x)^k.$$
Using this, we easily get that $\overline s(P) = \sum_{j=1}^\infty \lambda_j \Big(s\big(\phi_j|_X\big)\Big)^k.$ Naturally,
$$
 \|P\|_{\mathcal P_I(^kX)} =\|P\|_{\mathcal P_N(^kX)} =\|\overline s(P)\|_{\mathcal P_N(^kX)}\le \sum_{j=1}^\infty |\lambda_j|\| \phi_j|_X \|^k.
$$


Now,
\begin{align*}
0 < \delta & = \big|Q(y) - \overline s(P)(y)\big|  = \big|\widetilde Q(y) - \overline s(P)(y)\big| \\
& \leq \left| \sum_{j=1}^\infty \lambda_j\big( s\big(\phi_j|_X\big)(y)  + \phi_j^\bot (y) \big)^k  - \lambda_j s\big(\phi_j|_X\big)(y)^k\right| \\
& \leq \sum_{j=1}^\infty |\lambda_j|\sum_{i=1}^k  \binom ki \big\| s\big(\phi_j|_X\big) \big\|^{k-i}  \| \phi_j^{\bot} \|^i  \\
& = \sum_{j=1}^\infty |\lambda_j|\Big( \big\| s\big(\phi_j|_X\big) \big\|  + \| \phi_j^\bot \| \Big)^k  - |\lambda_j|\big\| s\big(\phi_j|_X\big) \big\|^k \\
&= \sum_{j=1}^\infty |\lambda_j|\| \phi_j \|^k - \sum_{j=1}^\infty |\lambda_j|\| \phi_j|_X \|^k\\
& \leq \|P\|_{\mathcal P_I(^kX)} + \frac{\delta}{2} - \|P\|_{\mathcal P_I(^kX)} = \frac{\delta}{2}.
\end{align*}
This is a contradiction. Thus, the result follows.
\end{proof}

Since $M$-embedded spaces are Asplund  we have a neater statement in this case:

\begin{corollary}
Let $X$ be an  $M$-ideal in  $X^{**}$.  If $P\in\mathcal P_I(^kX)$ then the Aron-Berner extension $\overline P$ is the unique norm preserving extension  to $\mathcal P_I(^kX^{**})$.
\end{corollary}

It is known that in $\ell_\infty$  integral and nuclear polynomials do not coincide. Consider thus a non nuclear polynomial $P\in \mathcal P_I(^k\ell_\infty)$. By the fact that $c_0$ is an $M$-ideal in $\ell_\infty$ and the previous corollary, we derive that the restriction of $P$ to $c_0$ should have integral (equivalently, nuclear) norm strictly smaller than the integral norm of $P$.

\bigskip

  As we commented before, if $X$ is Asplund (actually, in the more general case of $\ell_{1} \not\hookrightarrow  \widehat{\bigotimes}^{k,s}_{\varepsilon_{k,s}}X$) the spaces of nuclear  and integral polynomials concide isometrically.  So, this is the case for most of the classical Banach spaces.

The situation is quite different for Banach spaces containing $\ell_1$. We will see below that if $X$ is Banach space that contains a subspace isomorphic to $\ell_1$ and whose dual has the MAP, then the quotient space $\mathcal{P}_{I}(^kX)/\mathcal{P}_{N}(^kX)$ is non separable. To see this, first we look at the case of $X=\ell_1$. Using Theorem \ref{loc-compl} and arguing as in the proof of \cite[Prop. 2.4]{CPV} we can clarify  how the containment of $\mathcal{P}_{N}(^k\ell_{1})$ in $\mathcal{P}_{I}(^k\ell_{1})$ is.

\begin{proposition} \
The quotient space $ \mathcal{P}_{I}(^k\ell_{1})/\mathcal{P}_{N}(^k\ell_{1})$ contains a subspace isometric to $\ell_{\infty}/ c_{0}$. Moreover,
$\mathcal{P}_{N}(^k\ell_{1})$ is not complemented in a dual space.
\end{proposition}
\begin{proof}
The approximation property of $\ell_\infty$ gives us the equality $\mathcal{P}_{N}(^k\ell_{1})=\widehat{\bigotimes}^{k,s}_{\pi_{k,s}}\ell_{\infty}$. Also, since every $k$-homogeneous polynomial on $c_0$ is approximable, we have $\mathcal{P}_{I}(^k\ell_{1})=(\widehat{\bigotimes}^{k,s}_{\varepsilon_{k,s}}\ell_{1})^*=(\widehat{\bigotimes}^{k,s}_{\pi_{k,s}}c_{0})^{**}$. Thus, from \cite[Cor. 7]{CG} or Theorem \ref{loc-compl} in the following section, we obtain that $\mathcal{P}_{N}(^k\ell_{1})$ is a locally 1-complemented subspace of  $\mathcal{P}_{I}(^k\ell_{1})$ (see definition in next section). We can thus picture the following exact sequence (the image of each arrow coincides with the kernel of the next one):
$$
\begin{CD}
0@>>> \mathcal{P}_{N}(^k\ell_{1})=\widehat{\bigotimes}^{k,s}_{\pi_{k,s}}\ell_{\infty} @>{\Theta_{\pi_{k,s}}} >> \mathcal{P}_{I}(^k\ell_{1})=(\widehat{\bigotimes}^{k,s}_{\pi_{k,s}}c_{0})^{**}  @>q>> \mathcal{P}_{I}(^k\ell_{1})/\mathcal{P}_{N}(^k\ell_{1})@>>>0.\\
\end{CD}
$$
Consider the isometric embedding $\delta:\ell_{\infty} \to \mathcal{P}_{I}(^k\ell_{1})$ given by

\begin{equation*}
\phi \longmapsto \left((x_n)_n\in \ell_1 \longmapsto \delta(\phi)(x_n)=\sum_{n=1}^{\infty}\phi(n)x^k_{n} \right).\\
\end{equation*}
 It is clear that $\delta(c_{0})=\delta(\ell_{\infty})\cap\mathcal{P}_{N}(^k\ell_{1})$ and that $q\circ \delta$ factorizes through the quotient $\ell_{\infty}/c_{0}$. Then there is an isometric embedding  $\ell_{\infty}/c_{0}\hookrightarrow  \mathcal{P}_{I}(^k\ell_{1})/\mathcal{P}_{N}(^k\ell_{1})$  that makes commutative the following diagram

 $$
\begin{CD}
 0@>>>c_{0}@>>> \ell_{\infty} @>>>\ell_{\infty}/c_{0}@>>>0\\
& &@V\delta|_{c_{0}} VV @V\delta VV@V VV  \\
 0@>>>\mathcal{P}_{N}(^k\ell_{1}) @>>> \mathcal{P}_{I}(^k\ell_{1})@ >q>>\mathcal{P}_{I}(^k\ell_{1})/\mathcal{P}_{N}(^k \ell_{1})@>>>0.\\
\end{CD}
$$
Now, the proof of  \cite[Prop. 2.4]{CPV} can be adapted easily to our setting to obtain the desired result.
\end{proof}

For a Banach space $X$ whose dual has the MAP, the space of nuclear polynomials $\mathcal{P}_{N}(^k X)$ coincides with the projective tensor product $\widehat{\bigotimes}_{\pi_{k,s}}^{k,s}X^*$  and it is contained isometrically in the space of integral polynomials  $\mathcal{P}_{I}(^kX)=\big(\widehat{\bigotimes}_{\varepsilon_{k,s}}^{k,s}X\big)^*$ (see \cite[Duality Theorem]{F} or \cite[Cor. 3.4]{CarGal-five}). So, it makes sense to consider the quotient $\mathcal{P}_{I}(^kX)/\mathcal{P}_{N}(^k X)$.

\begin{corollary}  Let $X$ be a Banach space whose dual has the MAP. If  $X$ contains a copy of $\ell_{1}$, then the quotient  space $\mathcal{P}_{I}(^kX)/\mathcal{P}_{N}(^k X)$ is  non separable. \end{corollary}

\begin{proof}
We can suppose with no lost of generality, that $\ell_1$ is a subspace of $X$. Let $i:\ell_1\to X$ the inclusion. Since nuclear and integral polynomials are extendible, the restriction mappings $R_N:\mathcal{P}_{N}(^k X)\to \mathcal{P}_{N}(^k \ell_1)$ and $R_I:\mathcal{P}_{I}(^k X)\to \mathcal{P}_{I}(^k \ell_1)$ are surjective.

By the comment before the corollary, we have metric injections $J_X:\mathcal{P}_{N}(^k X)\to \mathcal{P}_{I}(^k X)$ and $J_{\ell_1}:\mathcal{P}_{N}(^k \ell_1)\to \mathcal{P}_{I}(^k \ell_1)$ that allow us to look at the quotient projections $q_X:\mathcal{P}_{I}(^k X)\to \mathcal{P}_{I}(^kX)/\mathcal{P}_{N}(^k X)$ and $q_{\ell_1}:\mathcal{P}_{I}(^k \ell_1)\to \mathcal{P}_{I}(^k\ell_1)/\mathcal{P}_{N}(^k \ell_1)$.

Thus, we have the diagram
  $$
\begin{CD}
   \mathcal{P}_{N}(^kX)  @>J_{X}>> \mathcal{P}_{I}(^kX) @>q_{X}>> \mathcal{P}_{I}(^kX)/\mathcal{P}_{N}(^kX) \\
@V R_N VV  @V R_I VV@V R VV  \\
 \mathcal{P}_{N}(^k\ell_{1}) @>J_{\ell_1} >> \mathcal{P}_{I}(^k\ell_{1}) @>q_{\ell_1}>> \mathcal{P}_{I}(^k\ell_{1})/\mathcal{P}_{N}(^k\ell_{1}),\\
\end{CD}\\
$$
where the left side square is easily seen to be commutative and the down arrow $R$ is defined  to make the right side square  commutative also. This fact and the surjectivity of $R_I$ and $q_{\ell_1}$  imply that the operator $R$ is onto. Since, by the previous proposition, the image of $R$ is not separable the same should be true for its domain $\mathcal{P}_{I}(^kX)/\mathcal{P}_{N}(^k X)$.
\end{proof}
%

\section{The bidual of a symmetric tensor product}

In order to obtain a characterization of polynomials belonging to an ideal that have unique norm preserving extension, we need to relate the bidual of the symmetric tensor  product of a Banach space with the  symmetric tensor product of its bidual.

More precisely, for $\beta_k$ an s-tensor norm of order $k$,
we study the relationship between the spaces $\widehat{\bigotimes}_{\beta_k}^{k,s}X^{**}$ and $\left(\widehat{\bigotimes}_{\beta_k}^{k,s}X\right)^{**}$. We need the notion of a local complementation.

\begin{definition}\cite{K} \label{def:local-comp}
Let $X$ be a subspace of $Y$ through $i$. We say that $X$ is
locally complemented  in $Y$ (through $i$) if there exists a constant $\lambda > 0$ such that for every  finite dimensional subspace $F\subset Y$  there exists and  operator $r_{F}:F \to X$, $\|r_{F}\|\leq \lambda$, such
that $r_{F}(i(x))=x$ whenever  $i(x)\in F$.
For the quantitative version we will say locally
$\lambda$-complemented.
\end{definition}
It is known that $X$ is locally complemented  in $Y$ (through $i: X \hookrightarrow Y$) if and
only if  $i^*:Y^*\to X^*$ has a left-inverse of bound $\lambda$ (i.e. $X^*$ is $\lambda$-complemented in $Y^*$). The
Principle of Local Reflexivity  of Lindenstrauss and Rosenthal \cite{LR}  says that every Banach space is locally complemented in its bidual. Also, it is well-known that every Banach space is locally
complemented in its ultrapowers.

In this terminology, it is clear that the Aron-Berner extension  ensures that the subspace $\widehat{\bigotimes}_{\beta_k}^{k,s}X$ is locally complemented in $\widehat{\bigotimes}_{\beta_k}^{k,s}X^{**}$ (through $\otimes^{k,s}\kappa_X : \widehat{\bigotimes}_{\beta_k}^{k,s}X \hookrightarrow  \widehat{\bigotimes}_{\beta_k}^{k,s}X^{**}$), then these spaces have the same local structure. Some questions arise: \\
\begin{itemize}
\item
Do $\widehat{\bigotimes}^{k,s}_{\beta_k}X^{**}$ and $(\widehat{\bigotimes}^{k,s}_{\beta_k}X)^{**}$ have the same local structure? \\
The answer is yes if $X^{**}$ has the bounded approximation property (Theorem \ref{loc-compl}).  \\

\item
Is there an application that makes commutative the following diagram?
\begin{eqnarray*}
\xymatrix{{\widehat{\bigotimes}_{\beta_k}^{k,s}X} \ar[rr]^{\otimes^{k,s}\kappa_X} \ar[dr]_{\kappa_{\big(\widehat \bigotimes^{k,s}_{\beta_k}X\big)}} & & {\widehat{\bigotimes}_{\beta_k}^{k,s}X^{**}}  \ar[ld]^{?} \\
&{ \left(\widehat{\bigotimes}_{\beta_k}^{k,s}X\right)^{**}} &  }
\end{eqnarray*}

The answer is always yes and the canonical application is given  by

\begin{eqnarray*}
\Theta_{\beta_k}:\widehat{\bigotimes}_{\beta_k}^{k,s}X^{**}&\longrightarrow
& \left(\widehat{\bigotimes}_{\beta_k}^{k,s}X\right)^{**}=\mathcal P_{\beta_k}(^kX)^*\\
v &\longmapsto &\big(P\mapsto \langle\overline{P},v\rangle\big).\\
\end{eqnarray*}
\end{itemize}

Since the Aron-Berner extension preserves the ideal norm for maximal polynomial ideals \cite[Cor. 3.4]{CarGal}, it is clear that $\|\Theta_{\beta_k}\|=1$. We want to study when $\Theta_{\beta_k}$ is an isomorphic embedding. As often happens when dealing with tensor
products, approximation properties will play a crucial role in
our proofs.
We adapt some techniques  developed in \cite{CG} for the
projective full tensor product $\widehat{\bigotimes}^{k}_{\pi_k}X$ to the
case of a general s-tensor norm $\beta_k$.

Cabello and Garc\'ia proved in \cite{CG} that if $X$ is a Banach space whose
bidual has the BAP, then $\Theta_{\pi}$ embeds $\widehat{\bigotimes}_{\pi_k}^{k}X^{**}$ as a locally
complemented subspace of $\left(\widehat{\bigotimes}_{\pi_k}^{k}X\right)^{**}$. Equivalently,  $\mathcal{L}(^{k}X^{**})$ is a complemented subspace of
$\mathcal{L}(^{k}X)^{**}$ for all $k\geq 1$.

We propose here a version of  this result for symmetric tensor products endowed with any $s$-tensor norm $\beta_k$. To prove it, we need first a lemma.\

\begin{lemma}\label{bapduals}
Let $\beta_k$ be an s-tensor norm of order $k$. If $X^{**}$ has the $\lambda$-AP, then  $\widehat{\bigotimes}^{k,s}_{\beta_k}X^{**}$ has the $\lambda^k$-AP. Moreover, there is a net of
finite rank operators $(t_\gamma)_\gamma$ with $t_\gamma:X\to X^{**}$ such that the operators   $T_\gamma =
\otimes^{k,s}t_\gamma^{**}:
\widehat{\bigotimes}^{k,s}_{\beta_k}X^{**}\to
\widehat{\bigotimes}^{k,s}_{\beta_k} X^{**}$ satisfy $\|T_\gamma\|\leq \lambda^k$
and
$$
\lim_\gamma \beta_k \big( T_\gamma (v)- v \big)=0,\qquad \textrm{for all }
v\in \widehat{\bigotimes}^{k,s}_{\beta_k} X^{**}.
$$
\end{lemma}
\begin{proof}
By \cite[Cor. 1]{CG}, the $\lambda$-AP of $X^{**}$ can be realized by mappings $t_\gamma^{**}$ where $t_\gamma:X\to X^{**}$ are finite rank operators satisfying $\|t_\gamma\|\le\lambda$.
This implies that the operators $T_\gamma$ have finite rank and  $\|T_\gamma\|=\|t_\gamma^{**}\|^k\leq \lambda^k$.

Now, to see that the net $\{T_\gamma\}_\gamma$ approximates the
identity, let us begin with an elementary tensor
$\otimes^k z\in \bigotimes^{k,s}_{\beta_k}X^{**}$. We
have,

\begin{eqnarray*}
\beta_k \big( T_\gamma(\otimes^k z)-\otimes^k z\big) & \leq &
\pi_{k,s} \big( T_\gamma(\otimes^k z)-\otimes^k z\big) \\
&=& \sup_{P\in
B_{\mathcal{P}(^kX^{**})}}\big|\langle P,T_\gamma(\otimes^k z)-\otimes^k z\rangle\big|\\
&=& \sup_{P\in
B_{\mathcal{P}(^kX^{**})}}\big| P(t_\gamma^{**}z)-P(z)\big|\\
&=& \sup_{P\in B_{\mathcal{P}(^kX^{**})}} \left| \sum_{j=0}^{k-1}
\binom{k}{j} \overset\vee P \big(z^j,(t_\gamma^{**}z-z)^{k-j}\big)\right|\\
&\leq & \sup_{P\in B_{\mathcal{P}(^kX^{**})}} \sum_{j=0}^{k-1}
\binom{k}{j} \|\overset\vee P\|\cdot \|z\|^j\cdot
\|t_\gamma^{**}z-z\|^{k-j}\longrightarrow 0.
\end{eqnarray*}
As a consequence, for any finite sum $u=\sum_{i=1}^N \lambda_i \otimes^k z_i$
it holds $\beta_k \big( T_\gamma(u)-u \big)\to 0$. Finally, since an
arbitrary $v\in \widehat{\bigotimes}^{k,s}_{\beta_k} X^{**}$ can be
approximated by finite sums of elementary tensors and the norms of
the operators $T_\gamma$ are bounded, we get that
$\beta_k \big( T_\gamma(v)-v \big) \to 0$.

\end{proof}

\begin{theorem}\label{loc-compl}
 Let $\beta_k$ be an s-tensor norm of order $k$. If $X^{**}$ has the $\lambda$-AP, then
$\Theta_{\beta_k}$ embeds $\widehat{\bigotimes}^{k,s}_{\beta_k}X^{**}$ as a locally $\lambda^k$-complemented subspace of
$\left(\widehat{\bigotimes}^{k,s}_{\beta_k}X\right)^{**}$. Equivalently, $\mathcal{P}_{\beta_k}(^kX^{**})$ is a
$\lambda^k$-complemented subspace of $\mathcal{P}_{\beta_k}(^kX)^{**}$.
\end{theorem}

\begin{proof}
We have already observed that  $\|\Theta_{\beta_k}(v)\|\leq \beta_k(v)$ for all $v\in
\widehat{\bigotimes}^{k,s}_{\beta_k} X^{**}$. Consider, as in Lemma \ref{bapduals}, a net of finite rank operators $t_\gamma: X\to X^{**}$. Thus, the net
$T_\gamma =
\otimes^{k,s}t_\gamma^{**}$ transfers the
BAP to $\widehat{\bigotimes}^{k,s}_{\beta_k}X^{**}$, with constant
$\lambda^k$.
In view of \cite[Lemma 4]{CG}, the proof will be completed if we show that for each
$\gamma$ there is an operator $\widetilde{T}_{\gamma}$ making commutative the diagram

\begin{equation}\label{diagram}
\xymatrix{
{\widehat{\bigotimes}_{\beta_k}^{k,s}X^{**}} \ar[rr]^{\Theta_{\beta_k}} \ar[dr]_{T_{\gamma }} & & {\left(\widehat{\bigotimes}_{\beta_k}^{k,s}X\right)^{**}} \ar[ld]^{\widetilde{T}_{\gamma}} \\
& { \widehat{\bigotimes}_{\beta_k}^{k,s}X^{**}}  &  }
\end{equation}

Let us consider the finite rank operators
$$
\widetilde{T}_\gamma = (\otimes^{k,s}t_\gamma)^{**}:
\left(\widehat{\bigotimes\nolimits}^{k,s}_{\beta_k}X\right)^{**}\longrightarrow
\widehat{\bigotimes\nolimits}^{k,s}_{\beta_k}X^{**}.
$$ Note that the ranges of the operators $\widetilde{T}_\gamma$ are in
$\widehat{\bigotimes\nolimits}^{k,s}_{\beta_k}X^{**}$ instead of
$\left(\widehat{\bigotimes\nolimits}^{k,s}_{\beta_k}X^{**}\right)^{**}$
because the mappings $\otimes^{k,s}t_\gamma$ have
finite rank. It is clear that $\| \widetilde{T}_\gamma \| \leq \lambda^k $.\

We have the following factorization (see (\ref{diagram})):
$$
T_\gamma =\widetilde{T}_\gamma \circ \Theta_{\beta_k}.
$$
Indeed, by linearity it is enough to prove that $T_\gamma
(\otimes^k z)=\widetilde{T}_\gamma \big( \Theta_{\beta_k}(\otimes^k z)\big)$, for any $z\in X^{**}$. To see this, let $Q\in
\mathcal{P}_{\beta_k}(^kX^{**})$, then
\begin{eqnarray*}
\langle Q,\widetilde{T}_\gamma \big( \Theta_{\beta_k}(\otimes^k z)\big)\rangle & = &
\langle Q,(\otimes^{k,s}t_\gamma)^{**} \big(
\Theta_{\beta_k}(\otimes^k z)\big)\rangle = \Theta_{\beta_k}(\otimes^k z) \left(
(\otimes^{k,s}t_\gamma)^{*}(Q)\right) \\
&=& \Theta_{\beta_k}(\otimes^k z) \left( Q\circ \otimes^{k,s}t_\gamma\right)
=\overline{Q\circ \otimes^{k,s}t_\gamma}(z)\\
&=& Q(t_\gamma ^{**} z) = \langle Q,T_\gamma (\otimes^k z)\rangle.
\end{eqnarray*}

\end{proof}

Note that the fact that  $\Theta_{\beta_k}$ embeds $\widehat{\bigotimes}^{k,s}_{\beta_k}X^{**}$ as a locally complemented subspace of
$\Big(\widehat{\bigotimes}^{k,s}_{\beta_k}X\Big)^{**}$, implies that $\Theta_{\beta_k}$ is an isomorphic embedding. Specifically we obtain that, if $X^{**}$ has the $\lambda$-AP, then, for all $v\in \widehat{\bigotimes}^{k,s}_{\beta_k}X^{**}$,
$$
\lambda^{-k} \beta_k(v)\leq \|\Theta_{\beta_k}(v)\|\leq \beta_k(v).
$$

Consequently, when $X^{**}$ has the MAP we derive the following corollary.

\begin{corollary} \label{MAP}
Let $\beta_k$ be an s-tensor norm of order $k$.
If $X^{**}$ has the MAP, then $\Theta_{\beta_k}$ is an isometry with its image.
\end{corollary}

\medskip
For the  case of an injective s-tensor norm of order $k$, $\beta_k$, the thesis of Corollary \ref{MAP} holds  without the hypothesis of $X^{**}$ having the approximation property.

\begin{proposition}\label{isometria}
Let  $\beta_k$ be an injective s-tensor norm of order $k$. Then, $$\Theta_{\beta_k}:\widehat{\bigotimes}_{\beta_k}^{k,s}X^{**}\rightarrow  \left(\widehat{\bigotimes}_{\beta_k}^{k,s}X\right)^{**}$$ is an isometry with its image.
\end{proposition}

To prove it, we need first the following lemma.

\begin{lemma}\label{diagcomm}
Let $\beta_k$ be an s-tensor norm of order $k$ and let $T \in \mathcal{L}(X,Y)$. Then, the following diagram commutes
\begin{equation}
\xymatrix{ \widehat{\bigotimes}_{\beta_k}^{k,s}X^{**} \ar[rr]^{\Theta_{\beta_k}} \ar[d]^{\otimes^{k,s}T^{**}} & & {\left(\widehat{\bigotimes}_{\beta_k}^{k,s}X\right)^{**}} \ar[d]^{(\otimes^{k,s}T)^{**}} \\
 {\widehat{\bigotimes}_{\beta_k}^{k,s}Y^{**}} \ar[rr]^{\Theta_{\beta_k}} & & \left(\widehat{\bigotimes}_{\beta_k}^{k,s}Y\right)^{**}  }
\end{equation}
\end{lemma}

\begin{proof}
It is  enough to prove that the diagram commutes when applied to elementary tensors. So, we need to see that
$\Theta_{\beta_k} \big( \otimes^{k,s}T^{**}  (\otimes^k z) \big) =   (\otimes^{k,s}T)^{**} \big( \Theta_{\beta_k} (\otimes^k z) \big)$, for all  $z \in X^{**}$.
Let $P \in \mathcal{P}_{\beta_k}(^kY) = \left(\widehat{\bigotimes}_{\beta_k}^{k,s}Y\right)^{*}$, then
$$\Theta_{\beta_k} \big( \otimes^{k,s}T^{**}  (\otimes^k z) \big) (P) = \Theta_{\beta_k} \big( \otimes^k T^{**}z \big) (P) = \langle \overline{P},\otimes^k T^{**}z \rangle = \overline{P}(T^{**}z).$$
On the other hand,
\begin{eqnarray*} (\otimes^{k,s}T)^{**} \big( \Theta_{\beta_k} (\otimes^k z) \big) (P) &=& \langle \Theta_{\beta_k} (\otimes^k z), (\otimes^k T)^{*}(P) \rangle =
\langle \overline{(\otimes^k T)^{*}(P)}, \otimes^k z  \rangle \\ &=& \langle \overline{P \circ \otimes^k T}, \otimes^k z  \rangle = \overline{P} (T^{**}z).
\end{eqnarray*}

\end{proof}

\begin{proof}(of Proposition~\ref{isometria})
Consider the canonical isometric inclusions:
$$ i_X : X \to C(B_{X^{*}}), \;\;\;\; i_X^{**} : X^{**} \to C(B_{X^{*}})^{**}.$$

Now, by Lemma~\ref{diagcomm}, the diagram commutes:
\begin{equation*}
\xymatrix{ \widehat{\bigotimes}_{\beta_k}^{k,s}X^{**} \ar[rr]^{\Theta_{\beta_k}} \ar[d]^{\otimes^{k,s}i_X^{**}} & & {\left(\widehat{\bigotimes}_{\beta_k}^{k,s}X\right)^{**}} \ar[d]^{(\otimes^{k,s}i_X)^{**}} \\
 {\widehat{\bigotimes}_{\beta_k}^{k,s}C(B_{X^{*}})^{**}} \ar[rr]^{\Theta_{\beta_k}} & & \left(\widehat{\bigotimes}_{\beta_k}^{k,s}C(B_{X^{*}})\right)^{**}  }.
\end{equation*}
By Lemma 4.4, Corollary 1 of 23.2 and Corollary 1 of 21.6 in \cite{DF} we know that $C(B_{X^{*}})^{**}$ has the MAP. Thus, by Corollary \ref{MAP}, $\Theta_{\beta_k}:\widehat{\bigotimes}_{\beta_k}^{k,s}C(B_{X^{*}})^{**}\to \left(\widehat{\bigotimes}_{\beta_k}^{k,s}C(B_{X^{*}})\right)^{**} $ is an isometry. Since $\beta_k$ is an injective s-tensor norm, the mappings $\otimes^{k,s}i_X^{**}$ and $(\otimes^{k,s}i_X)^{**}$ are also isometries. Finally, the commutativity of the diagram yields the isometry of the desired mapping $\Theta_{\beta_k}:\widehat{\bigotimes}_{\beta_k}^{k,s}X^{**}\to \left(\widehat{\bigotimes}_{\beta_k}^{k,s}X\right)^{**}$.
\end{proof}

\begin{question}
 If $\Theta_{\beta_k}$ is an isomorphic embedding, does it imply  that  $\Theta_{\beta_k}$ embeds $\widehat{\bigotimes}^{k,s}_{\beta_k}X^{**}$ as a locally complemented subspace of $(\widehat{\bigotimes}^{k,s}_{\beta_k}X)^{**}$?
\end{question}

\begin{remark}\rm
Another result for 2-fold tensor products of spaces without AP was proved in \cite[Cor. 4]{CG} by  Cabello and Garc\'ia. They showed that $\Theta_{\pi}:X^{**}\widehat \otimes_{\pi} X^{**}\to (X\widehat\otimes_{\pi} X)^{**}$ is an isomorphic embedding  if $X$ has type $2$ and $X^*$ has cotype $2$. Examples of this result are $\mathcal K(\mathcal H)$, $\mathcal L(\mathcal H)$
and a Pisier space \cite{Pis} having no uniformly
complemented finite dimensional subspaces. Note that Pisier space
and $\mathcal L(\mathcal H)= \mathcal K(\mathcal H)^{**}$  both fail
the AP.

A canonical commutative diagram allows us to translate this statement to the symmetric case. Thus, we have: if $X$ has type $2$ and $X^*$ has cotype $2$, then $\Theta_{\pi_{2,s}}:\widehat{\bigotimes}_{\pi_{2,s}}^{2,s} X^{**}\rightarrow  \left(\widehat
{\bigotimes}_{\pi_{2,s}}^{2,s}X\right)^{**}$ is an isomorphic embedding.
\end{remark}

\begin{remark} \rm \textbf{The holomorphic case}\label{hbbeta}

The previous results have some consequences in the holomorphic setting. To state them, we need to consider `the same' s-tensor norm $\beta_k$ of order $k$, for each $k$. It is clear what we mean by `the same' when $\beta_k=\pi_{k,s}$ or $\beta_k=\varepsilon_{k,s}$, for all $k$. For the general case we refer to the concept of `coherent sequence of polynomial ideals' defined in \cite{CaDiMu} or \cite{CaDiMu2}.

Hence, we say that $\beta=(\beta_k)_k$ is an \emph{s-tensor norm} if, for each $k$, $\beta_k$ is an s-tensor norm of order $k$ and the sequence $\{\mathcal P_{\beta_k}\}_k$ is coherent.

Following \cite{CaDiMu2}, to a given s-tensor norm $\beta$, we can associate a Fr\'{e}chet space of holomorphic functions of bounded type:
$$
H_{b\beta}(X)=\left\{f\in H(X)\ : \frac{d^kf(0)}{k!} \in \mathcal P_{\beta_k}(^kX)\textrm{
for all } k\textrm{ and }
\limsup_{k\to\infty}\Big\|\frac{d^kf(0)}{k!}\Big\|_{\mathcal P_{\beta_k}(^kX)}^{\frac1k}
= 0 \right\}.
$$

Analogously, we can consider holomorphic functions of bounded type associated to $\beta$ defined on an open ball $B_r(x)$ of center $x$ and radius $r$:
$$
H_{b\beta}(B_r(x))=\left\{f\in H(B_r(x))\ : \frac{d^kf(x)}{k!} \in \mathcal P_{\beta_k}(^kX)\textrm{
for all } k\textrm{ and }
\limsup_{k\to\infty}\Big\|\frac{d^kf(x)}{k!}\Big\|_{\mathcal P_{\beta_k}(^kX)}^{\frac1k}
\leq\frac1r \right\}.
$$

Galindo, Maestre and Rueda in \cite{GMR} introduced and developed the concept of `$R$-Schauder decomposition'. For $0<R\leq\infty$, a sequence of Banach spaces $(E_k,\|\cdot\|_k)$ is an \emph{$R$-Schauder decomposition} of a Fr\'{e}chet space $E$ if it is a Schauder decomposition and verifies the condition: for every sequence $(x_k)_k$, with $x_k\in E_k$, the series $\sum_{k=1}^\infty x_k$ converges in $E$ if and only if $$\limsup_k\|x_k\|^{\frac1k}_k\leq\frac1R.$$

For an s-tensor norm  $\beta$ we know from \cite[Prop. 3.2.11 and Prop 3.2.53]{Mu} the following:
\begin{itemize}
\item $\{\mathcal P_{\beta_k}(^kX)\}_k$ is an $\infty$-Schauder decomposition of $H_{b\beta}(X)$.
\item $\{\mathcal P_{\beta_k}(^kX)\}_k$ is an $r$-Schauder decomposition of $H_{b\beta}(B_r(x))$.
\end{itemize}

Now, from Theorem \ref{loc-compl} and Corollary \ref{MAP}, invoking some results of \cite{GMR} and arguing as in \cite[Th. 4]{CG} we obtain, for any s-tensor norm $\beta$, the following:
\begin{itemize}
\item If $X^{**}$ has the BAP, then $H_{b\beta}(X^{**})$ is a complemented subspace of $H_{b\beta}(X)^{**}$.
\item If $X^{**}$ has the MAP, then $H_{b\beta}(B^{**}_r(x))$ is a complemented subspace of $H_{b\beta}(B_r(x))^{**}$ (where $B^{**}_r(x)$ means the ball of $X^{**}$ with center $x\in X$ and radius $r$).
\end{itemize}

\end{remark}
\section{Unique norm preserving extension for a polynomial belonging to an ideal}

Godefroy  gave in \cite{G} a characterization of norm-one functionals having unique norm
preserving extensions to the bidual as the points of $S_{X^*}$ where the identity is $w ^*$- $w$ continuous (see also
\cite[Lemma III.2.14]{HWW}).

\begin{lemma}
 Let $X$ be a Banach space and $x^{\ast }
\in S_{X^{\ast }}$.  The following are equivalent:
\begin{itemize}
\item[$(i)$]  $x^{\ast }$ has a unique norm preserving extension to a  functional on $X^{\ast \ast }$.
\item[$(ii)$] The function $Id_{B_{X^{\ast }}}: (B_{X^{\ast }} ,w^{\ast }) \longrightarrow (B_{X^{\ast }} ,w) $
is continuous at $x^{\ast }$.
\end{itemize}
\end{lemma}


Aron, Boyd and Choi  presented in \cite{ABCh} a polynomial version of this result:

\begin{proposition} \label{equi-ABCh}
Let $X$ be a Banach space such that $X^{\ast\ast }$ has the MAP and let $P\in S_{\mathcal{P}(^{k}X)}$. The
following are equivalent:
\begin{itemize}
\item[$(i)$]  $P$ has a unique norm preserving extension to  $\mathcal{P}(^{k}X^{\ast\ast })$.
\item[$(ii)$] If $\{P_{\alpha}\}_\alpha\subset  B_{\mathcal{P}(^{k}X)}$ converges pointwise to $P$, then $\{\overline{P_{\alpha}}\}_\alpha$
converges pointwise to
 $\overline{P}$ in $X^{\ast\ast }$.
\end{itemize}
\end{proposition}

We are interested on having a similar characterization for unique
norm preserving extensions to the bidual of polynomials belonging to
a maximal (scalar-valued) polynomial ideal. In this case, obviously, the norm
that we want to preserve is the ideal norm.

Let $\beta_k$ be an $s$-tensor norm of order $k$ and let $P\in
\mathcal{P}_{\beta_k}(^kX)$.
Since we have already mentioned that  the Aron-Berner
extension preserves the ideal norm for maximal polynomial ideals \cite{CarGal}, should $P$ has unique norm preserving extension, this extension ought to be
$\overline{P}$.

To prove our result we need the following equivalence
between different topologies for the convergence of nets of
polynomials in an ideal unit ball. The proof is straightforward.

\begin{lemma}\label{convergencia}
Suppose that the polynomial $P$ and the net $\{P_\alpha\}_\alpha$
are contained in the unit ball of $\mathcal{P}_{\beta_k}(^kX)$, where $\beta_k$ is an s-tensor norm of order $k$.
Then, the following are equivalent:
\begin{enumerate}
\item[$(i)$] $P_\alpha(x)\to P(x)$ for all $x\in X$.
\item[$(ii)$] $P_\alpha\to P$ for the topology
$\sigma(\mathcal{P}_{\beta_k}(^kX),\widehat{\bigotimes}_{\beta_k}^{k,s}X)$.
\item[$(iii)$] $P_\alpha\to P$ for the topology
$\sigma(\mathcal{P}(^kX),\widehat{\bigotimes}_{\pi_{k,s}}^{k,s}X)$.
\end{enumerate}
\end{lemma}

\begin{theorem}\label{godefroy}
Let $\beta_k$ be an s-tensor norm of order $k$ and suppose that
 $X^{**}$ has the MAP. Consider a polynomial
$P\in\mathcal{P}_{\beta_k}(^kX)$ with
$\|P\|_{\mathcal{P}_{\beta_k}(^kX)}=1$. Then, the following are
equivalent:
\begin{enumerate}
\item[$(i)$] $P$ has a unique norm preserving extension to  $\mathcal{P}_{\beta_k}(^kX^{**})$.
\item[$(ii)$] The $\beta_k$-Aron-Berner extension $(AB)_{\beta_k}:
\left( B_{\mathcal{P}_{\beta_k}(^kX)}, w^*\right)\longrightarrow \left(
B_{\mathcal{P}_{\beta_k}(^kX^{**})}, w^*\right)$ is continuous at
$P$.
\item[$(iii)$] If the net $\{P_{\alpha}\}_\alpha\subset  B_{\mathcal{P}_{\beta_k}(^{k}X)}$ converges pointwise to $P$,
then $\{\overline{P_{\alpha}}\}_\alpha$ converges pointwise to
 $\overline{P}$ in $X^{\ast\ast }$.
\end{enumerate}
\end{theorem}

\begin{proof}
$(i)\Rightarrow (ii)$. Let $\{P_{\alpha}\}_\alpha\subset
B_{\mathcal{P}_{\beta_k}(^{k}X)}$ such that
$P_\alpha\overset{w^*}{\rightarrow} P$. We want to see that
$\overline{P}_\alpha\overset{w^*}{\rightarrow} \overline{P}$ in
$\mathcal{P}_{\beta_k}(^kX^{**})$. By the compactness of $\left(
B_{\mathcal{P}_{\beta_k}(^kX^{**})}, w^*\right)$, the net
$\{\overline{P_{\alpha}}\}_\alpha$ has a subnet
$\{\overline{P_{\gamma}}\}_\gamma$ $w^*$-convergent to a polynomial
$Q\in B_{\mathcal{P}_{\beta_k}(^kX^{**})}$.

For each $x\in X$, we have, on one hand, that
$\overline{P}_\gamma(x)=P_\gamma(x)\to P(x)$ and, on the other hand,
that $\overline{P}_\gamma(x)\to Q(x)$. So, $Q|_X=P$. Also,
$\|Q\|\leq 1=\|P\|$ implies
$\|Q\|_{\mathcal{P}_{\beta_k}(^kX^{**})}=\|P\|_{\mathcal{P}_{\beta_k}(^kX)}$.
This means that $Q$ is a norm preserving extension of $P$ and by
$(i)$ it should be $Q=\overline{P}$. Since for every subnet of
$\{P_{\alpha}\}_\alpha$ we can find a sub-subnet such that the
Aron-Berner extensions are $w^*$-convergent to $\overline{P}$, we
conclude that $\overline{P}_\alpha\overset{w^*}{\rightarrow}
\overline{P}$.

$(ii)\Rightarrow (i)$. Let $Q\in \mathcal{P}_{\beta_k}(^kX^{**})$ an
extension of $P$ with $\|Q\|_{\mathcal{P}_{\beta_k}(^kX^{**})}=1$.
From Corollary \ref{MAP},
$\Theta_{_{\beta_k}}:\widehat{\bigotimes}_{\beta_k}^{k,s}X^{**}\longrightarrow
 \left(\widehat{\bigotimes}_{\beta_k}^{k,s}X\right)^{**}$ is an
 isometry. Due to this, the polynomial $Q\in \mathcal{P}_{\beta_k}(^kX^{**})=
 \left(\widehat{\bigotimes}_{\beta_k}^{k,s}X^{**}\right)^*$ has a
 Hahn-Banach extension $\widetilde{Q}\in
 \left(\widehat{\bigotimes}_{\beta_k}^{k,s}X\right)^{***}=
 \mathcal{P}_{\beta_k}(^kX)^{**}$. By Goldstine, there exist a net $\{P_{\alpha}\}_\alpha\subset
B_{\mathcal{P}_{\beta_k}(^{k}X)}$ such that
$P_\alpha\overset{w^*}{\rightarrow} \widetilde{Q}$, where $w^*$
means here the topology $\sigma(\mathcal{P}_{\beta_k}(^kX)^{**},
\mathcal{P}_{\beta_k}(^kX)^{*})$.

Let $u\in \widehat{\bigotimes}_{\beta_k}^{k,s}X\subset
\mathcal{P}_{\beta_k}(^kX)^{*}$. So we have
$$
\langle P_\alpha , u\rangle \rightarrow \langle \widetilde{Q},
u\rangle = \langle Q, u\rangle = \langle P, u\rangle.
$$
This means that $P_\alpha\overset{w^*}{\rightarrow} P$, where $w^*$
denotes here the topology $\sigma(\mathcal{P}_{\beta_k}(^kX),
\widehat{\bigotimes}_{\beta_k}^{k,s}X)$. By $(ii)$, this implies that
$\overline{P}_\alpha\to \overline{P}$ for the topology
$\sigma(\mathcal{P}_{\beta_k}(^kX^{**}),
\widehat{\bigotimes}_{\beta_k}^{k,s}X^{**})$.

Now, if $v\in \widehat{\bigotimes}_{\beta_k}^{k,s}X^{**}$, it follows
that
$$
\langle \overline{P}_\alpha , v\rangle \rightarrow \langle
\overline{P}, v\rangle.
$$
But also, since $v\in  \mathcal{P}_{\beta_k}(^kX)^{*}$,
$$
\langle \overline{P}_\alpha , v\rangle = \langle v, P_\alpha\rangle
\rightarrow \langle v, \widetilde{Q}\rangle = \langle Q, v\rangle.
$$
Therefore, $\overline{P}=Q$.

The equivalence between $(ii)$ and $(iii)$ is a consequence of the
previous lemma.
\end{proof}

For the particular case of $\beta_k$ being an
injective $s$-tensor norm, the same argument but applying Proposition \ref{isometria} instead of Corollary \ref{MAP}, yields to a version of
Theorem \ref{godefroy} without the hypothesis of metric
approximation property.

\begin{corollary}
Let $\beta_k$ be an injective s-tensor norm of order $k$. Consider a polynomial
$P\in\mathcal{P}_{\beta_k}(^kX)$ with
$\|P\|_{\mathcal{P}_{\beta_k}(^kX)}=1$.
Then, the following are equivalent:
\begin{enumerate}
\item[$(i)$] $P$ has a unique norm preserving extension to  $\mathcal{P}_{\beta_k}(^kX^{**})$.
\item[$(ii)$] The ``$\beta_k$''-Aron-Berner extension $(AB)_{\beta_k}:
\left( B_{\mathcal{P}_{\beta_k}(^kX)}, w^*\right)\longrightarrow \left(
B_{\mathcal{P}_{\beta_k}(^kX^{**})}, w^*\right)$ is continuous at $P$.
\item[$(iii)$] If the net $\{P_{\alpha}\}_\alpha\subset  B_{\mathcal{P}_{\beta_k}(^{k}X)}$ converges pointwise to $P$,
then $\{\overline{P_{\alpha}}\}_\alpha$ converges pointwise to
 $\overline{P}$ in $X^{\ast\ast }$.
\end{enumerate}

\end{corollary}

\section{ $Q_{\beta}$-reflexivity}
In \cite{AD}   Aron and  Dineen considered the problem of obtaining a polynomial functional representation of the bidual of $\mathcal{P}(^kX)$. More precisely, they asked when the space $\mathcal{P}(^kX)^{**}$ is isomorphic to
$\mathcal{P}(^kX^{**})$ in a canonical way. Spaces
with this property are called $Q$-reflexive. A reflexive Banach space $X$ with the
approximation property is $Q$-reflexive if and only if
$\mathcal{P}(^kX)$ is reflexive, for all $k$.

In an analogous way, we consider a sort of $Q$-reflexivity for a maximal polynomial ideal. Recall that a maximal polynomial ideal is the dual of the symmetric tensor product endowed with a finitely generated s-tensor norm (see Remark \ref{hbbeta} for the definition of an s-tensor norm $\beta=(\beta_k)_k$).

\begin{definition} Let $\beta$ be an s-tensor norm. We say that a Banach space $X$ is $Q^k_{{\beta}}$-reflexive if
 $\Theta_{\beta_k}^*:\mathcal{P}_{\beta_k}(^kX)^{**}\longrightarrow \mathcal{P}_{\beta_k}(^kX^{**})$ is an isomorphism. We say that $X$ is $Q_{{\beta}}$-reflexive if it is $Q^k_{{\beta}}$-reflexive for all $k$.

\end{definition}

Naturally, $Q_{{\pi_{s}}}$-reflexive spaces are the usual $Q$-reflexive spaces.
It follows from \cite[Th. 4]{JPZ} that  if  $X ^{**}$ has the AP and the Radon-Nikod\'{y}m property (RNP),  then  $X$ is $Q^k_{{\pi_{s}}}$-reflexive if  and only if $\mathcal{P}(^kX)=\mathcal{P}_{w}(^kX)$ (for the multilinear case see \cite[Cor. 3]{CG}).  We obtain here a kind of ``predual version'' of that result,  a condition for the $Q^k_{\varepsilon_{s}}$-reflexivity.

\begin{theorem}\label{Q-reflexive}  Suppose that $X^{**}$  has the AP and $X^*$ has  the RNP. Then, $X$ is $Q^k_{\varepsilon_{s}}$-reflexive if  and only if $\mathcal{P}(^kX^*)=\mathcal{P}_{w}(^kX^*)$.
\end{theorem}

\begin{proof}
First we consider the Borel transformation $B_{k}:\mathcal{P}(^kX^*)^* \to \mathcal{P}_{I}(^kX^{**})$ on $X^*$. $B_{k}$ is the adjoint of the natural operator $\widehat{\bigotimes}^{k,s}_{\varepsilon_{k,s}}X^{**} \to \big(\widehat{\bigotimes}^{k,s}_{\pi_{k,s}}X^{*}\big)^*= \mathcal{P}(^kX^*)  $ (which is always an isometric embedding).\\

Now, the RNP of $X^*$ implies that $\mathcal{P}_{I}(^kX)=\mathcal{P}_{N}(^kX)$ \cite{BR,CD}, and the AP of $X^*$ yields the equality $\mathcal{P}_{N}(^kX)=\widehat{\bigotimes}^{k,s}_{\pi_{k,s}}X^{*}$. Then,
$$
\mathcal{P}_{I}(^kX)^{**}=\mathcal{P}_{N}(^kX)^{**}= \Big(\widehat{\bigotimes}^{k,s}_{\pi_{k,s}}X^{*}\Big)^{**}=\mathcal{P}(^kX^*)^*.
$$

It follows that $\mathcal{P}_{I}(^kX)^{**}$ is isomorphic to $\mathcal{P}_{I}(^kX^{**})$ through $\Theta_{\varepsilon_{k,s}}^*$ if and only if $B_{k}$ is an isomorphism.\\

On the other hand, if $X^{**}$ has the approximation property then every polynomial in $\mathcal{P}_{w}(^kX^*)$ is
uniformly approximable on $B_{X^*}$ by finite type polynomials and hence $\mathcal{P}_{w}(^kX^*)$ is isometrically isomorphic to $\widehat{\bigotimes}^{k,s}_{\varepsilon_{k,s}}X^{**}$ \cite{AP}. Hence,
$$
\mathcal{P}_{w}(^kX^*)^*= \Big(\widehat{\bigotimes}^{k,s}_{\varepsilon_{k,s}}X^{**}\Big)^{*}=\mathcal{P}_{I}(^kX^{**}).
$$

Therefore, $B_{k}$ is an isomorphism if and only if $\mathcal{P}(^kX^*)=\mathcal{P}_{w}(^kX^*)$.
\end{proof}

In the above result, the RNP on $X^*$ can be replaced by weaker hypothesis that $\ell_{1}$ is not a subspace of $\widehat{\bigotimes}^{k,s}_{\varepsilon_{k,s}}X$ (see \cite{BR,CD}).

\begin{remark}\label{epsilon-pi}\rm
 If $\mathcal{P}_{I}(^kX^{**})= \widehat{\bigotimes}^{k,s}_{\pi_{k,s}}X^{***}  $ (for example if $X^{***}$ has the AP and the RNP), then  $X$ is $Q^k_{\varepsilon_{s}}$-reflexive if  and only if  $X ^*$ is $Q^k_{\pi_{s}}$-reflexive. Indeed, if $X ^*$ is $Q^k_{\pi_{s}}$-reflexive, one has

$$
\mathcal{P}_{I}(^kX^{**})=\widehat{\bigotimes}^{k,s}_{\pi_{k,s}}X^{***}\simeq \Big(\widehat{\bigotimes}^{k,s}_{\pi_{k,s}}X^{*}\Big)^{**}=\mathcal{P}_{I}(^kX)^{**}.
$$
On the other side, being $X$ a $Q^k_{\varepsilon_{s}}$-reflexive space, we have
$$
\widehat{\bigotimes}^{k,s}_{\pi_{k,s}}X^{***}=\mathcal{P}_{I}(^kX^{**})\simeq \mathcal{P}_{I}(^kX)^{**}=\Big(\widehat{\bigotimes}^{k,s}_{\pi_{k,s}}X^{*}\Big)^{**}.
$$
Note that in both cases we use that whether the equality $\mathcal{P}_{I}(^kX^{**})= \widehat{\bigotimes}^{k,s}_{\pi_{k,s}}X^{***}$ holds, it is also valid that $\mathcal{P}_{I}(^kX)= \widehat{\bigotimes}^{k,s}_{\pi_{k,s}}X^{*}$.
\end{remark}

\begin{example}
\begin{itemize}
\item[$(i)$]  The  Tsirelson space $T$ and the Tsirelson-James space $T_{J}$ are  $Q_{{\varepsilon_{s}}}$-reflexive  (note that $\mathcal{P}(^kT^*)=\mathcal{P}_{w}(^kT^*)$ and $\mathcal{P}(^kT^*_{J})=\mathcal{P}_{w}(^kT_{J}^*)$   \cite{AD}). The space $X=\widehat{\otimes}^{2,s}_{\varepsilon_{2,s}}T_{J} $  is  $Q_{{\varepsilon_{s}}}$-reflexive   but not quasi-reflexive (i.e., $X^{**}/X$ is infinite-dimensional). The $Q_{{\varepsilon_{s}}}$-reflexivity of $X$ is a consequence of Remark \ref{epsilon-pi} and the fact, proved in \cite{Go}, that $X^*=\widehat{\otimes}^{2,s}_{\pi_{2,s}}T_{J}^* $ is $Q_{{\pi_{s}}}$-reflexive.
\item[$(ii)$] The space $X=\ell_{p}$ is  $Q^k_{\varepsilon_{s}}$-reflexive if and only if   $k<p^*$ (where $p^*$ is the conjugate of $p$).  The space $X=L_{p}$ ($1<p<\infty$) contains a complemented copy of $\ell_{2}$. Thus, $L_{p}$ is not $Q^k_{\varepsilon_{s}}$-reflexive for any $k>1$.
\end{itemize}

\end{example}
\medskip
As appeared in the proof of Theorem \ref{Q-reflexive}, we can translate the statement of that theorem in the following way: if $X^{**}$ has the AP and $X^*$ has the RNP, then $X$ is $Q^k_{\varepsilon_{s}}$-reflexive if  and only if $\mathcal{P}(^kX^*)=\widehat{\bigotimes}^{k,s}_{\varepsilon_{k,s}}X^{**}$. Now we present a similar result for the case of the symmetric tensor norm $\beta=/ \pi_{s} \s$. Recall that $\left(\widehat{\bigotimes}^{k,s}_{\et}X\right)^*=\mathcal{P}_{e}(^kX)$ is the ideal of extendible $k$-homogeneous polynomials.

\begin{theorem}\label{Q-reflexive eta}  Let $X$ be a Banach space such that $X^{*}$  has the BAP and the RNP. Then, $X$ is $Q^k_{/ \pi_{s} \s}$-reflexive if and only if $\mathcal{P}_{\etd}(^kX^*)=\widehat{\bigotimes}^{k,s}_{\et}X^{**}$.
\end{theorem}

\begin{proof}
Since $X$ is Asplund and $X^*$ has the BAP then, by the comments after \cite[Cor. 2.4]{CarGal-SRN}, we have $\mathcal{P}_{e}(^kX)=\widehat{\bigotimes}^{k,s}_{\etd}X^{*}$. More precisely, the mapping
$$J_{\etd}: \widehat{\bigotimes}^{k,s}_{\etd}X^{*} \to \mathcal{P}_{e}(^kX),$$
given by $J_{\etd}(\o \phi)= \phi^k$  is an isometric isomorphism.
Thus, $\mathcal{P}_{e}(^kX)^{*}= \left( \widehat{\bigotimes}^{k,s}_{\etd}X^{*} \right)^{*}=\mathcal P_{\etd}(^k X^{*})$ via $(J_{\etd})^*$.

Consider the following two mappings
$$\Theta_{\et}: \widehat{\bigotimes}^{k,s}_{\et}X^{**} \to \mathcal{P}_{e}(^kX)^{*}\qquad\textrm{and}$$
$$J_{\et} : \widehat{\bigotimes}^{k,s}_{\et}X^{**} \to  \left( \widehat{\bigotimes}^{k,s}_{\etd}X^{*} \right)^{*}=\mathcal P_{\etd}(^k X^{*}),$$
where $\Theta_{\et}(\o z)(P)= \overline P(z)$ and $J_{\et}(\o z)= z^k$.

Notice that, with the above identifications, these two mappings are equal.
Indeed, since $J_{\etd} ( \bigotimes^{k,s} X^*)$ is dense in $\mathcal{P}_{e}(^kX)$, by linearity  we only have to check the following equality
\begin{align*}
\Theta_{\et}(\o z)(J_{\etd}(\o \phi)) & = \Theta_{\et}(\o z)(\phi^k) \\
& = \overline{\phi^k} (z) = (z(\phi))^k \\
& =z^k(\phi) = J_{\et}(\o z)(\phi).
\end{align*}
In consequence, we obtain that $X$ is $Q^k_{/ \pi_{s} \s}$-reflexive if and only if $J_{\et}$ is an isomorphism. And this is equivalent to the identity $\mathcal{P}_{\etd}(^kX^*)=\widehat{\bigotimes}^{k,s}_{\et}X^{**}$.
\end{proof}

Carando and the second author  studied in \cite{CarGal-SRN} the symmetric Radon-Nikod\'{y}m property. For adjoints of projective tensor norms having this property the previous result is canonically extended:

\begin{theorem}\label{Q-reflexive beta}  Let $\beta_k$ be a projective s-tensor norm of order $k$ with the symmetric Radon-Nikod\'{y}m property and let $X$ be an Asplund space such that $X^{*}$ has the BAP. Then, $X$ is $Q^k_{\beta^*}$-reflexive if and only if $\mathcal{P}_{\beta_k}(^kX^*)=\widehat{\bigotimes}^{k,s}_{\beta_k^*}X^{**}$.
\end{theorem}

\subsection*{Acknowledgements}
We would like to thank Daniel Carando, F\'{e}lix Cabello S\'{a}nchez and Jes\'{u}s Su\'{a}rez for  helpful comments and suggestions.

\end{document}